\documentclass[envcountsame]{llncs}
\usepackage{amsmath}
\usepackage{amsfonts}
\usepackage{amssymb}
\usepackage{epsfig}

\DeclareMathOperator{\Frob}{Frob}

\DeclareMathOperator{\rank}{rank}

\newcommand{\inv}{^{-1}}

\DeclareFontEncoding{OT2}{}{} % to enable usage of cyrillic fonts

\newcommand{\ra}{\rightarrow}

\newcommand{\cO}{\mathcal{O}}

 \DeclareMathOperator{\Div}{Div}

\newcommand{\comment}[1]{}
\newcommand{\Q}{\mathbb{Q}}

\newcommand{\C}{\mathbb{C}}
\newcommand{\Cp}{{\mathbb C}_p}

\renewcommand{\P}{\mathbb{P}}

\newcommand{\Z}{\mathbb{Z}}
\newcommand{\F}{\mathbb{F}}

\newcommand{\Qp}{\Q_p}

\newcommand{\softO}{\widetilde{O}}

\spnewtheorem{algorithm}[theorem]{Algorithm}{\bfseries}{\itshape}

\newcounter{listnum}

\title{Explicit Coleman Integration for Hyperelliptic Curves}
\author{Jennifer S. Balakrishnan\inst{1} \and Robert W. Bradshaw\inst{2} \and Kiran S. Kedlaya\inst{1}}
\institute{Massachusetts Institute of Technology, Cambridge, MA 02139, USA; \email{jen@math.mit.edu},
\email{kedlaya@mit.edu}
\and
University of Washington, Seattle, WA 98195, USA; \email{robertwb@math.washington.edu} }

\begin{document}

\maketitle{}

\begin{abstract}
Coleman's theory of $p$-adic integration figures prominently
in several number-theoretic applications, such as finding torsion and rational points on curves,
and computing $p$-adic regulators in $K$-theory (including $p$-adic heights on elliptic curves).
We describe an algorithm for computing Coleman integrals on hyperelliptic curves,
and its implementation in Sage.
\end{abstract}

\section{Introduction}

One of the fundamental difficulties of $p$-adic analysis is that the totally disconnected
topology of $p$-adic spaces makes it hard to introduce a meaningful form of antidifferentiation.
It was originally discovered by Coleman that this problem can be circumvented using the
principle of \emph{Frobenius equivariance}. Using this idea, Coleman introduced a $p$-adic
integration theory first on the projective line \cite{coleman:dilogarithms},
then (partly jointly with de Shalit) on curves and abelian varieties
\cite{coleman:torsion}, \cite{coleman-deshalit}.
Alternative treatments have been given by Besser \cite{besser:coleman} using
methods of $p$-adic cohomology, and by Berkovich \cite{berkovich:integration} using
the nonarchimedean Gel'fand transform.

Although Coleman's construction is in principle quite suitable for machine computation,
this had only been implemented previously in the genus 0 case
\cite{besser-dejeu}.
The purpose of this paper is to present an algorithm for computing single 
Coleman integrals on hyperelliptic curves of good reduction over $\Cp$ for $p>2$,
based on the third author's algorithm for computing the Frobenius action on the
de Rham cohomology of such curves \cite{kedlaya:mw}.
We also describe an implementation of this algorithm in the Sage
computer algebra system. 

For context, we indicate some of the many potential applications
of explicit Coleman integration. Some of these will be treated, with additional
numerical examples, in the first author's upcoming PhD thesis. (Some of these
applications will require additional refinements of our implementation; see Section~\ref{sec6}.)

\begin{itemize}
\item
\textit{Torsion points on curves.}
Coleman's original application of $p$-adic integration was to find torsion points on
curves of genus greater than 1. This could potentially be made effective and automatic.

\item
\textit{$p$-adic heights on curves.}
Investigations into $p$-adic analogues of the conjecture of Birch and Swinnerton-Dyer for Jacobians of hyperelliptic curves require computation of the Coleman-Gross height pairing \cite{coleman-gross}. This global $p$-adic height pairing can, in turn, be decomposed into a sum of local height pairings at each prime. In particular, for $C$ a hyperelliptic curve over $\Q_p$ with $p$ a prime of good reduction and for $D_1,D_2 \in \Div^0(C)$ with disjoint support, the Coleman-Gross $p$-adic height pairing at $p$ is given in terms of the Coleman integral \cite{coleman:torsion} $$h_p(D_1,D_2) = \int_{D_2} \omega_{D_1},$$ for an appropriately constructed
differential $\omega_{D_1}$ associated to the divisor $D_1$. This  pairing is effectively
computable by work of the first author \cite{jsb:heights}. 

Using this work, it should be
possible (using ideas of Besser \cite{besser:sage}) to add in local heights away from $p$, and 
thus compute the Coleman-Gross height pairing on Jacobians of hyperelliptic curves. 
(In genus 1, one can then compare to an alternate computation
based on work of Mazur-Stein-Tate \cite{mazur-stein-tate} and Harvey \cite{harvey:heights}.)

\item
\textit{$p$-adic regulators.}
A related topic to the previous one
is the computation of $p$-adic regulators in higher $K$-theory of arithmetic
schemes, which are expected to relate to special values of $L$-functions. Some computations
in genus 0 have been made by Besser and de Jeu \cite{besser-dejeu}.

\item
\textit{Rational points on curves: Chabauty's method.}
For $C$ a smooth proper curve over $\Z[\frac{1}{N}]$, the 
\emph{Chabauty condition} on $C$ is that $\rank
J(C)\left(\Z\left[\frac{1}{N}\right]\right) < \dim J(C),$ where
$J(C)$ denotes the Jacobian of the curve. When the Chabauty condition holds,
there exists a 1-form $\omega$ on $J(C)^{\mathrm{an}}$ with
$\int_0^P \omega = 0$ for all points $P \in
J(C)\left(\Z\left[\frac{1}{N}\right]\right)$. We might be able to
compute $C(\Z[\frac{1}{N}])$ if we can find all points $P\in
C^{\mathrm{an}}$ such that $\int_0^P\omega =0$.
This method has already been used in many cases, by Coleman and many others; 
see \cite{mccallum-poonen} for a survey (circa 2007). 

To apply Chabauty's method in a typical case, one needs the integral of $\omega$
at some point in a residue disc, with which one can find all zeroes of the integral
in the residue disc. Several methods are suggested in \cite[Remark~8.3]{mccallum-poonen}
for doing this, including Coleman integration. However, no serious attempt has been made
to use numerical Coleman integration in Chabauty's method; it seems likely that it can
handle cases where the other methods suggested in 
\cite[Remark~8.3]{mccallum-poonen} for finding constants of integration
prove to be impractical.

\item
\textit{Rational points on curves: nonabelian Chabauty.}
It may be possible to use (iterated) Coleman integration to find rational points on curves
failing the Chabauty condition, using Kim's
nonabelian Chabauty method \cite{kim:chabauty}. As a demonstration of the method,
Kim \cite{kim:rank1} 
gives an explicit double integral which vanishes
on the integral points of the minimal regular model of
a genus 1 curve over $\Q$ of Mordell-Weil rank 1.
The erratum to \cite{kim:rank1} includes a corrected formula, together with some
numerical examples computed using the methods of this paper.

\item
\textit{$p$-adic polylogarithms and multiple zeta values.}
These have been introduced recently by Furusho \cite{furusho}, but little numerical data
exists so far.
\end{itemize}

\subsubsection*{Acknowledgments.}
The authors thank William Stein for access to his computer \texttt{sage.math.washington.edu}
(funded by NSF grant DMS-0821725), and Robert Coleman and Bjorn Poonen for helpful conversations.
Balakrishnan was supported by a National Defense Science and Engineering Graduate Fellowship
and an NSF Graduate Research Fellowship.
Bradshaw was supported by NSF grant DMS-0713225.
Kedlaya was supported by NSF CAREER grant DMS-0545904,
the MIT NEC Research Support Fund,
and the MIT Cecil and Ida Green Career Development Professorship.
Some development work was carried out at the 2006 MSRI Summer Graduate Workshop
on computational number theory, and the 2007 Arizona Winter School on $p$-adic geometry.

\section{Coleman's theory of $p$-adic integration}\label{sec1}

In this section, we recall Coleman's $p$-adic integration theory (for single integrals only)
in the case of curves with good reduction.
This theory involves some concepts from rigid analytic geometry which it would be hopeless
to introduce in such limited space; some standard references are
\cite{bgr} and \cite{fresnel-vanderput}. (See also \cite[\S 1]{coleman:torsion}.)

Let $\C_p$ be a completed algebraic closure of $\Q_p$, and let $\cO$ be the valuation
subring of $\C_p$.
Choose once and for all a \emph{branch of the $p$-adic logarithm}, i.e., a homomorphism
$\mathrm{Log}: \C_p^\times \to \C_p$ whose restriction to the disc 
$\{x \in \C_p: |x-1| < 1\}$ is given by the logarithm series
$\log(x) = \sum_{i=1}^\infty (1-x)^i/i$.
(The choice of branch has no effect on the integrals on differentials of the second kind,
i.e., everywhere meromorphic differentials with all residues zero.)

We first introduce integrals on discs and annuli within $\P^1$.
\begin{definition} \label{D:usual integration}
Let $I$ be an open subinterval of $[0, +\infty)$.
Let $A(I)$ denote the annulus (or disc)
$\{t \in \mathbb{A}^1_{\C_p}: |t| \in I\}$.
For $\sum_{i \in \Z} c_i t^i\,dt \in \Omega^1_{A(I)/\C_p}$ and $P,Q \in A(I)$,
define
\[
\int_P^Q \sum_{i \in \Z} c_i t^i\,dt = c_{-1} \mathrm{Log}(Q/P) + \sum_{i\neq -1} \frac{c_i}{i+1} (Q^{i+1} - P^{i+1}).
\]
This is easily shown not to depend on the choice of the coordinate $t$.
\end{definition}
\begin{remark}
Note that because of the division by $i+1$ in the formula for the integral, we are unable to
integrate on \emph{closed} discs or annuli.
\end{remark}

We next turn to curves of good reduction.
\begin{definition}
By a \emph{curve} over $\cO$, we will mean a smooth proper connected
scheme $X$ over $\cO$ of relative dimension $1$. Equip the function field
$K(X)$ with the $p$-adic absolute value, so that the elements of $K(X)$ of norm at most $1$
constitute the 
local ring in $X$ of the generic point of the special fibre $\overline{X}$ of $X$.

Let $X_{\Q}$ denote the generic fibre of $X$ as a rigid analytic space.
There is a natural specialization map from $X_{\Q}$ to 
$\overline{X}$; the inverse image of any point of $\overline{X}$ is a subspace
of $X_{\Q}$ isomorphic to an open unit disc. We call such a disc a \emph{residue disc}
of $X$.
\end{definition}

\begin{definition}
Let $X$ be a curve over $\cO$.
By a \emph{wide open subspace} of $X_{\Q}$, we will mean a rigid analytic subspace of $X_{\Q}$
of the form $\{x \in X_{\Q}: |f(x)| < \lambda\}$ for some $f \in K(X)$
of absolute value $1$ and some $\lambda >1$.
\end{definition}

Coleman made the surprising discovery that there is a well-behaved 
integration theory on wide open subspaces
of curves over $\cO$, exhibiting no phenomena of path dependence.
(Note that one needs to consider wide open subspaces even to integrate
differentials which are holomorphic or meromorphic
on the entire curve.)
In the case of hyperelliptic curves,
Coleman's construction of these integrals using Frobenius lifts will be reflected
in our technique for computing the integrals.
For the general case, see \cite[\S 2]{coleman:torsion}, \cite[\S 4]{besser:coleman}, or \cite[Theorem~1.6.1]{berkovich:integration}. 

\begin{theorem}[Coleman] \label{thm:coleman}
We may assign to each curve $X$ over $\cO$ and each wide open subspace $W$ of $X_\Q$
a map $\mu_W: \mathrm{Div}^0(W) \times \Omega^1_{W/\C_p} \to \C_p$, subject to
the following conditions. (Here $\mathrm{Div}(W)$ denotes the free
group on the elements of $W$, and $\mathrm{Div}^0(W)$ denotes the kernel of the degree map
$\deg: \mathrm{Div}(W) \to \Z$ taking each element of $W$ to $1$.)
\begin{enumerate}
\item[(a)] (Linearity)
The map $\mu_W$ is linear on $\mathrm{\Div}^0(W)$ and $\C_p$-linear on $\Omega^1_{W/\C_p}$.
\item[(b)] (Compatibility)
For any residue disc $D$ of $X$
and any isomorphism $\psi: W \cap D \to A(I)$ for some interval $I$,
the restriction of $\mu_W$ to $\mathrm{Div}^0(W \cap D) \times \Omega^1_{W/\C_p}$
is compatible with Definition~\ref{D:usual integration} via $\psi$.
\item[(c)] (Change of variables)
Let $X'$ be another curve over $\cO$, let $W'$ be a wide open subspace of $X'$,
and let $\psi: W \to W'$ be any morphism of rigid spaces relative to an automorphism of $\C_p$.
Then
\begin{equation} \label{eq:change of variables}
\mu_{W'}(\psi(\cdot), \cdot) = \mu_{W}(\cdot, \psi^*(\cdot)).
\end{equation}
\item[(d)] (Fundamental theorem of calculus)
For any $Q = \sum_i c_i (P_i) \in \mathrm{Div}^0(W)$ and any $f \in \mathcal{O}(W)$,
$\mu_{W}(Q, df) = \sum_i c_i f(P_i)$.
\end{enumerate}
\end{theorem}

\begin{remark}
One cannot expect path independence in the case of bad reduction. 
For instance, an elliptic curve over $\C_p$ with bad 
reduction admits a Tate uniformization, so its logarithm map has nonzero periods in general.
In Berkovich's theory of integration, this occurs because the nonarchimedean analytic
space associated to this curve $X$ has nontrivial first homology.
\end{remark}

\section{Explicit integrals for hyperelliptic curves}\label{sec3}

We now specialize to the situation where $p>2$ and $X$ is a genus $g$ hyperelliptic curve over 
an unramified extension $K$ of $\Qp$ having good reduction. 
We will assume in addition that we have been given a model of $X$ of the form $y^2 = f(x)$ such that $\deg f(x) = 2g+1$ and $f$ has no repeated roots modulo $p$. 
(This restriction is inherited from \cite{kedlaya:mw}, where it is used to simplify
the reduction procedure. One could reduce to this case 
after possibly replacing $K$ by a larger unramified extension of $\Qp$, 
by performing a linear fractional transformation in $x$ to put one root at infinity, 
thus reducing the degree from $2g+2$ to $2g+1$.)
We will distinguish between \emph{Weierstrass} and \emph{non-Weierstrass} residue discs of $X$,
which respectively correspond to Weierstrass and non-Weierstrass points of $\overline{X}$.

To discuss the differentials we will be integrating, we review
a core definition from \cite{kedlaya:mw}.
Let $X'$ be the affine curve obtained by deleting the Weierstrass points from $X$, 
and let $A = K[x,y,z]/(y^2 - f(x),yz -1)$ be the coordinate ring of $X'$.
\begin{definition}The Monsky-Washnitzer (MW) weak completion of $A$ is the ring $A^{\dagger}$
consisting of infinite sums of the form $$\left\{\sum_{i = -\infty}^{\infty} \frac{B_i(x)}{y^i},\; B_i(x) \in K[x], \deg B_i \leq 2g\right\},$$ further subject to the condition that $v_p(B_i(x))$ grows faster than a linear function of $i$ as $i \ra \pm \infty$.
We make a ring out of these using the relation
$y^2 = f(x)$.
\end{definition} 
These functions are holomorphic on wide opens, so we will integrate 1-forms \begin{equation}\label{omega}\omega = g(x,y)\frac{dx}{2y},\quad g(x,y) \in A^{\dagger}. 
\end{equation}
Note that we only consider 1-forms which are \emph{odd}, i.e., which are negated by the
hyperelliptic involution. Even 1-forms can be written in terms of $x$ alone, and so can
be integrated directly as in Definition~\ref{D:usual integration}.
(This last statement would fail if we had taken $A^\dagger$ to be 
the full $p$-adic completion of $A$,
rather than the weak completion. This observation is the basis for Monsky-Washnitzer's
formal cohomology, which is used in \cite{kedlaya:mw}.)

Note that the class of allowed forms includes those meromorphic differentials on $X$ whose poles all belong to Weierstrass residue
discs. For some applications (e.g., $p$-adic canonical heights), it is necessary to integrate meromorphic differentials with poles in
non-Weierstrass residue discs. These will be discussed in \cite{jsb:heights}.

Note also that for ease of exposition, we describe all of our algorithms as if it were possible to compute exactly
in $A^\dagger$. This is not possible for two reasons: the elements of $A^\dagger$ correspond to infinite series, and
the coefficients of these series are polynomials with $p$-adic coefficients. In practice, each computation will be made with
suitable $p$-adic approximations of the truly desired quantities, so one must keep track of how much $p$-adic precision
is needed in these estimates in order for the answers to bear a certain level of $p$-adic accuracy. We postpone this discussion to
\S~\ref{precision}.

\subsection{A basis for de Rham cohomology}
\label{de rham}

We first note that any odd differential $\omega$ as in \eqref{omega} can be written
uniquely as
\begin{equation} \label{basis diffs}
\omega = df + c_0\omega_0 + \cdots +
c_{2g-1}\omega_{2g-1}
\end{equation} 
with $f \in A^\dagger$, $c_i \in K$, and
\begin{equation} \label{wi}
\omega_i = \frac{x^i\,dx}{2y} \qquad (i=0,\dots,2g-1).
\end{equation}
That is, the $\omega_i$ form a 
basis of the odd part of the de Rham cohomology of $A^{\dagger}$.
The process of putting $\omega$ in the form \eqref{basis diffs}, 
using the relations
\begin{align*}
y^2 &= f(x),\\
% 2ydy &= f'(x)dx,\\
 d(x^iy^j) &= \left(2ix^{i-1}y^{j+1}+ jx^if'(x)y^{j-1} \right) \frac{dx}{2y},
\end{align*}
can be made algorithmic; see \cite[\S 3]{kedlaya:mw}.
(Briefly, one uses the first relation to reduce high powers of $x$, and the second to 
reduce large positive and negative powers of $y$.)
Using properties from Theorem~\ref{thm:coleman} 
(linearity and the fundamental theorem of calculus),
the integration of $\omega$ reduces effectively to the integration of the $\omega_i$.

It may be convenient for some purposes to use a different basis of de Rham cohomology.
For instance, the basis $x^i\,dx/2y^3 \,(i=0,\dots,2g-1)$ is \emph{crystalline}
(see the erratum to \cite{kedlaya:mw}), so Frobenius will act via a matrix
with $p$-adically integral entries.

\subsection{Tiny integrals}

We refer to any Coleman
integral of the form $\int_P^Q \omega$ in which $P,Q$ lie in the same residue disc
(Weierstrass or not)
as a \emph{tiny integral}.
As an easy first case, we give an algorithm to compute tiny integrals of basis differentials.
\begin{algorithm}[Tiny Coleman integrals]\label{tiny} \;\\
%Let $P, Q \in X(\C_p)$ be points in the same residue disc, neither equal to the point at infinity.
\textbf{Input:} Points $P, Q \in X(\C_p)$ in the same residue disc (neither equal to the point at infinity) and a basis differential $\omega_i$. \\
\textbf{Output:} The integral $\int_P^Q \omega_i$.
\begin{enumerate}
\item Construct a linear interpolation from $P$ to $Q$. For instance, in a
non-Weierstrass residue disc, we may take
\begin{align*} x(t) &= (1-t)x(P) + tx(Q)\\
y(t) &= \sqrt{f(x(t))},\end{align*} where $y(t)$ is expanded as a formal power series in $t$.

\item Formally integrate the power series in $t$:$$\int_P^Q \omega_i = \int_P^Q x^i\frac{dx}{2y} = \int_0^1 \frac{x(t)^i}{2y(t)}\frac{dx(t)}{dt}dt.$$\end{enumerate}\end{algorithm}

\begin{remark} \label{more tiny}
One can similarly integrate any $\omega$ holomorphic in the residue disc containing $P$ and $Q$.
If $\omega$ is only meromorphic in the disc, but has no pole at $P$ or $Q$, we can
first make a polar decomposition, i.e., wrie $\omega$ as a holomorphic differential on the disc
plus some terms of the form $c/(t - r)^i$, and integrate the latter terms directly.
(If $\omega$ is everywhere meromorphic, this is achieved by a partial fractions decomposition.)
\end{remark}

\subsection{Non-Weierstrass discs}

We next compute integrals of the form $\int_P^Q \omega_i$
in which $P,Q \in X(\C_p)$ lie in distinct non-Weierstrass residue discs.
The method of tiny integrals is not available;
we instead employ Dwork's principle of analytic continuation along Frobenius,
in the form of Kedlaya's algorithm \cite{kedlaya:mw} for calculating the action of 
Frobenius on de Rham cohomology. 
Note that we calculate the integrals $\int_P^Q \omega_i$
for all $i$ simultaneously. (We modify the presentation in \cite{kedlaya:mw} by
keeping track of exact differentials, which are irrelevant for computing zeta functions.)

\begin{algorithm}[Kedlaya's algorithm]\label{kedlaya}\;\\
\textbf{Input:} The basis differentials $\{\omega_i\}_{i=0}^{2g-1}$.\\
\textbf{Output:} Functions $f_i \in A^{\dagger}$ and a $2g \times 2g$ matrix $M$ over $K$ such that $\phi^*(\omega_i) =  df_i + \sum_{j=0}^{2g-1}M_{ij} \omega_j$ for a $p$-power lift of Frobenius $\phi$.
\begin{enumerate}
\item 
Since $K$ is an unramified extension of $\Qp$,
it carries a unique automorphism $\phi_K$ lifting the 
Frobenius automorphism $x \mapsto x^p$ on its residue field.
Extend $\phi_K$ to a Frobenius lift on $A^{\dagger}$ by setting 
\begin{align*}
\phi(x) &= x^p,\\
\phi(y) &= y^p\left(1 + \frac{\phi_K(f)(x^p)-f(x)^p}{f(x)^p}\right)^{1/2}\\
&= y^p \sum_{i=0}^{\infty}\binom{1/2}{i}\frac{(\phi_K(f)(x^p)-f(x)^p)^i}{y^{2pi}},
\end{align*}
noting the series converges in $A^{\dagger}$ because $\phi_K(f)(x^p)-f(x)^p$ has positive valuation. 
(This choice of $\phi(y)$ ensures that
$\phi(y)^2 = \phi(f(x))$, so that the action on $A^{\dagger}$ is well-defined.
\item 
Use a Newton iteration to compute $y/\phi(y)$. Then for $i=0,\dots,2g-1$, proceed as in 
\S~\ref{de rham} to write
\begin{equation} \label{phi de rham}
\phi^*(\omega_i) = p x^{pi+p-1} \frac{y}{\phi(y)} \frac{dx}{2y} = df_i + \sum_{j=0}^{2g-1} 
M_{ij} \omega_j
\end{equation}
for some $f_i \in A^\dagger$ and some $2g \times 2g$ matrix $M$ over $K$.
\end{enumerate}
\end{algorithm}

We may use Algorithm~\ref{kedlaya} to compute Coleman integrals between endpoints in
non-Weierstrass residue discs, as follows. (Note that our recipe is essentially Coleman's
construction of the integrals in this case.)
\begin{algorithm}[Coleman integration in non-Weierstrass discs]\label{algo:nonteich}\;\\
\textbf{Input:} The basis differentials $\{\omega_i\}_{i=0}^{2g-1}$, points $P,Q \in X(\C_p)$ in non-Weierstrass residue discs, and a positive integer $m$ such that the residue fields of $P,Q$ are contained in $\mathbb{F}_{p^m}$.\\
\textbf{Output:} The integrals $\{ \int_P^Q \omega_i\}_{i=0}^{2g-1}$.
%Let $\omega$ be an odd differential on $X$.
%Let $P,Q \in X(\C_p)$ be points in non-Weierstrass residue discs. 
%Choose a positive integer $m$ such that the residue fields of $P$ and $Q$
%are contained in $\mathbb{F}_{p^m}$. 

\begin{enumerate}

\item 
Calculate the action of the $m$-th power of Frobenius on each basis element (see 
Remark~\ref{remark:frobenius power}):
\begin{equation} \label{frobenius power}
(\phi^m)^* \omega_i = df_i + \sum_{j=0}^{2g-1} M_{ij}\omega_j.
\end{equation}

\item 
By change of variables 
(see Remark~\ref{rem:nonteich}), we obtain 
\begin{equation}\label{linear}\sum_{j=0}^{2g-1}
  (M-I)_{ij}\int_P^Q \omega_j = f_i(P)-f_i(Q) -
\int_P^{\phi^m(P)}\omega_i -\int_{\phi^m(Q)}^Q\omega_i
\end{equation}
(the \emph{fundamental linear system}).
As the eigenvalues of the matrix $M$ are algebraic integers of 
$\C_p$-norm $p^{m/2} \neq 1$ (see \cite[\S 2]{kedlaya:mw}), the matrix $M-I$ is invertible, 
and we may solve \eqref{linear} to obtain the integrals $\int_P^Q \omega_i$.

\end{enumerate}
\end{algorithm}

\begin{remark} \label{remark:frobenius power}
To compute the action of $\phi^m$, first perform Algorithm~\ref{kedlaya} to write
$$\phi^* \omega_i = dg_i + \sum_{j=0}^{2g-1} B_{ij}\omega_j.$$ 
If we view $f,g$ as column vectors and $M,B$ as matrices, we then have
\begin{align*}
f &= \phi^{m-1}(g) + B \phi^{m-2}(g) + \cdots + B \phi_K(B) \cdots \phi_K^{m-2}(B) g \\
M &= B \phi_K(B) \cdots \phi_K^{m-1}(B).
\end{align*}
\end{remark}
 
\begin{remark} \label{rem:nonteich}
We obtain \eqref{linear} as follows.
By change of variables,\begin{align*}\int_{\phi^m(P)}^{\phi^m(Q)}\omega_i &= \int_{P}^{Q}(\phi^m)^*\omega_i\\
&= \int_{P}^{Q}(df_i + \sum_{j=0}^{2g-1}M_{ij}\omega_j)\\
&= f_i(Q)-f_i(P) + \sum_{j=0}^{2g-1}M_{ij}\int_P^Q \omega_j.\end{align*}
Adding $\int_{P}^{\phi^m(P)}\omega_i + \int_{\phi^m(Q)}^Q\omega_i$ to both sides of this equation yields $$\int_P^Q\omega_i = \int_P^{\phi^m(P)}\omega_i + \int_{\phi^m(Q)}^Q\omega_i + f_i(Q) - f_i(P) + \sum_{j=0}^{2g-1}M_{ij}\int_P^Q \omega_j,$$ which is equivalent to \eqref{linear}.
\end{remark}

\begin{definition}
A \emph{Teichm\"{u}ller point} of $X_\Q$ is a point fixed by some power of $\phi$.
Each non-Weierstrass residue disc contains a unique such point: if $(\overline{x}, \overline{y}) \in \overline{X}$
is a non-Weierstrass point, the Teichm\"{u}ller point in its residue disc has $x$-coordinate
equal to the usual Teichm\"uller lift of $x$. This leaves two choices for the $y$-coordinate,
exactly one of which has the correct reduction modulo $p$. Note that Teichm\"uller points
are always defined over finite \emph{unramified} extensions of $\Q_p$.
\end{definition}

\begin{remark}
A variant of Algorithm~\ref{algo:nonteich} is to first
find the Teichm\"uller points $P', Q'$ in the residue discs of $P,Q$,
then note that from the fundamental linear system \eqref{linear}, we have
\begin{equation}\label{linear2}
\sum_{j=0}^{2g-1}
(M-I)_{ij}\int_{P'}^{Q'} \omega_j = f_i(P')-f_i(Q').
\end{equation}
From \eqref{linear2}, we obtain the integrals $\int_{P'}^{Q'} \omega_i$. 
Finally, write $\int_P^Q \omega_i - \int_{P'}^{Q'} \omega_i$ as the sum 
$\int_P^{P'} \omega_i +
\int_{Q'}^Q \omega_i$ of tiny integrals.
\end{remark}

\subsection{Weierstrass endpoints of integration}\label{sec4}

Suppose now that $P,Q$ lie in different residue discs,
at least one of which is Weierstrass.
Since a differential $\omega$ of the form \eqref{omega} is not meromorphic over
Weierstrass residue discs, we cannot always even define $\int_P^Q \omega$,
let alone compute it.
We will thus assume (to cover most cases arising in applications)
that $\omega$ is everywhere meromorphic, with no pole at either 
$P$ or $Q$. We then make the following observation.
\begin{lemma}\label{w_to_nw}
Let $\omega$ be an odd, everywhere meromorphic differential on $X$. Choose $P,Q \in X(\C_p)$ which are not poles of 
$\omega$, with $P$ Weierstrass.
Then for $\iota$ the hyperelliptic involution,
$\int_P^Q \omega = \frac{1}{2}\int_{\iota(Q)}^Q \omega$. In particular,
if $Q$ is also a Weierstrass point, then $\int_P^Q \omega = 0$.\end{lemma}
\begin{proof}Let $I:=\int_P^Q \omega = \int_P^{\iota(Q)} (- \omega) = \int_{\iota(Q)}^P \omega$. Then by additivity in the endpoints, we have $\int_{\iota(Q)}^Q \omega = 2I$, from which the result follows.\end{proof}

If $P$ belongs to a Weierstrass residue disc while $Q$ does not,
we find the Weierstrass point $P'$ in the disc of $P$, then apply
Lemma~\ref{w_to_nw} to write
\begin{equation} \label{one weierstrass disc}
\int_P^Q \omega = \int_P^{P'} \omega + \frac{1}{2} \int_{\iota(Q)}^{Q} \omega.
\end{equation}
The first integral on the right side of \eqref{one weierstrass disc} is tiny,
while the second integral involves two points in non-Weierstrass residue discs,
and so may be computed as in the previous section.
The situation is even better if $P,Q$ both belong to residue discs containing
respective Weierstrass points $P',Q'$: in this case,
by Lemma~\ref{w_to_nw}, $\int_P^Q \omega$ equals the sum $\int_P^{P'} \omega
+ \int_{Q'}^Q \omega$ of tiny integrals.

\begin{remark} \label{rem:iterated}
Beware that Lemma~\ref{w_to_nw} does not generalize to iterated integrals.
For instance, for double integrals, if both integrands are odd, the total integrand
is even, so the argument of Lemma~\ref{w_to_nw} tells us nothing. It is thus worth
considering alternate approaches for dealing with Weierstrass discs, which may
generalize better to the iterated case. We concentrate on the case where 
$P$ lies in a Weierstrass residue disc but $Q$ does not, as we may reduce to this case
by splitting $\int_P^Q \omega = \int_P^R \omega + \int_R^Q \omega$ for some
auxiliary point $R$ in a non-Weierstrass residue disc.

In Algorithm~\ref{algo:nonteich}, the form $f_i$ belongs to $A^\dagger$ and so
need not converge at $P$. However, it does converge at any point $R$ near the boundary
of the disc, i.e., in the complement of a certain smaller disc which can be 
bounded explicitly. We may thus write $\int_P^Q \omega_i = \int_P^R \omega_i + 
\int_R^Q \omega_i$ for suitable $R$ in the disc of $P$, to obtain an analogue
of the fundamental linear system \eqref{linear}. Similarly, when we write
$\omega$ as in \eqref{basis diffs}, we can find $R$ close enough to
the boundary of the disc of $P$ so that $f$ converges at $R$, use \eqref{basis diffs}
to evaluate $\int_R^Q \omega$, then compute $\int_P^R \omega$ as a tiny integral.
One defect of this approach is that forcing $R$ to be close to the boundary of
the residue disc of $P$ forces $R$ to be defined over a highly ramified extension of 
$\Q_p$, over which computations are more expensive.

An alternate approach exploits the fact that for $P$ in the infinite residue 
disc but distinct from the point at infinity, we may compute $\int_P^Q \omega$
directly using Algorithm~\ref{algo:nonteich}.
This works because both the Frobenius lift and the reduction process
respect the subring of $A^\dagger$ consisting of functions which are meromorphic
at infinity. When $P$ lies in a finite Weierstrass residue disc, we may reduce to the 
previous case using a change of variables on the $x$-line to move $P$ 
to the infinite disc. However, one still must use the approach of the previous
paragraph to reduce evaluation of $\int_P^Q \omega$ to evaluation of the
$\int_P^Q \omega_i$.
\end{remark}

\section{Implementation notes and precision}\label{sec5}

We have implemented the above algorithms in Sage \cite{sage} for curves defined over
$\Q_p$. In doing so, we made the following
observations.

\subsection{Precision estimates}
\label{precision}

For a tiny integral, the precision of the result 
depends on the truncation of the power series computed. 
Here is the analysis for a non-Weierstrass disc; the analysis for a Weierstrass disc, 
using a different local interpolation, is similar. (For points over ramified 
extensions, one must also account for the ramification index in the bound,
but it should be clear from the proof how this is done.)

\begin{proposition} \label{prop:tiny precision}
Let $\int_P^Q \omega$ be a tiny integral in a non-Weierstrass residue disc, with $P,Q$ defined over
an unramified extension of $K$ and accurate to $n$ digits of precision. 
Let $(x(t),y(t))$ be the local interpolation between $P$ and $Q$ defined by
\begin{align*}x(t) &= x(P)(1-t) + x(Q)t = x(P) + t(x(Q)-x(P))\\
y(t) &= \sqrt{f(x(t))}.\end{align*}
Let $\omega = g(x,y)dx$ be a differential of the second kind such that
$h(t) = g(x(t), y(t))$ belongs to $\cO[[t]]$.
If we truncate $h(t)$ modulo $t^m$, then
the computed value of the integral $\int_P^Q \omega$ will be correct to 
$\min\{n, m+1- \lfloor \log_p (m+1) \rfloor\}$ digits of (absolute) precision.
\end{proposition}

\begin{proof}
 Let $t' = t(x(Q)-x(P))$.
 As $P,Q$ are in the same residue disc and are defined over an unramified extension
of $K$, we have $v_p(x(Q)-x(P)) \geq 1$. If we expand $g(x(t'),y(t')) = 
\sum_{i=0}^\infty c_i (t')^i$, then by hypothesis $c_i \in \cO$. Thus 
\begin{align*}\int_P^Q \omega &= \int_P^Q g(x,y)dx\\
&= \int_0^1 g(x(t),y(t))dx(t)\\
&= \int_0^{x(Q)-x(P)} g(x(t'),y(t')) dt'\\
&= \int_0^{x(Q)-x(P)} \sum_{i=0}^{\infty} c_i (t')^i dt'\\
&= \sum_{i=0}^\infty \frac{c_i}{i+1} (x(Q) - x(P))^{i+1}.
\end{align*}
The effect of omitting $c_i (t')^i$ 
from the expansion of $g(x(t'), y(t'))$ for some $i \geq m$ is to change the final
sum by a quantity of valuation at least $i+1 - \lfloor \log_p (i+1) \rfloor
\geq m+1 - \lfloor \log_p (m+1) \rfloor$. The effect of the ambiguity in $P$ and $Q$
is that the computed value of $(x(Q)-x(P))^{i+1}$ differs from the true value
by a quantity of valuation at least $i+1-\lfloor \log_p(i+1) \rfloor + n-1 \geq n$.
\end{proof}

For Coleman integrals between different residue discs,
which we may assume are non-Weierstrass thanks to \S~\ref{sec4},
one must first account for the precision loss in Algorithm~\ref{kedlaya}.
According to \cite[Lemmas~2,3]{kedlaya:mw} and the erratum to \cite{kedlaya:mw}
(or \cite{harvey:matfrob}),
working to precision $p^N$ in Algorithm~\ref{kedlaya} produces the $f_i, M_{ij}$
accurately modulo $p^{N-n}$ for $n = 1 + \lfloor \log_p \max\{N, 2g+1\} \rfloor$.

We must then take into account the objects involved in the linear system (\ref{linear}),
as follows.

\begin{proposition}\label{m}Let $\int_P^Q \omega$ be a Coleman integral, with $\omega$ a differential of the second kind and with $P,Q$ in non-Weierstrass residue discs, defined over an unramified extension of $\Q_p$, and accurate to $n$ digits of precision. 
Let $\Frob$ be the matrix %\footnote{Note that the matrix of Frobenius $\Frob$ is calculated to the precision as in \cite{harvey:matfrob}.}
of the action of Frobenius on the basis differentials.  
Set $B = \Frob^t - I$, and let $m = v_p(\det(B))$.  
Then the computed value of the integral $\int_P^Q \omega$ will 
be accurate to $n-\max\{m , \lfloor \log_p n \rfloor\}$ digits of precision.\end{proposition}

\begin{proof}By the linear system (\ref{linear}), 
the Coleman integral is expressed in terms of tiny integrals, 
integrals of exact forms evaluated at points, and a matrix inversion. 
Suppose that the entries of $B = \Frob^t - I$ are computed to precision $n$. 
Then taking $B\inv$, we have to divide by $\det(B)$, which lowers the precision by 
$m = v_p(\det(B))$. By Proposition~\ref{prop:tiny precision}, 
computing tiny integrals (with the series expansions truncated modulo $t^{n-1}$)
gives a result 
precise up to $n- \lfloor \log_p n\rfloor$ digits.
Thus the value of the integral $\int_P^Q \omega$ will be correct to $n-\max\{m , \lfloor \log_p n \rfloor\}$ digits of precision.\end{proof}

\subsection{Complexity analysis}

We assume that asymptotically fast integer and polynomial multiplication algorithms are used; specifically addition, subtraction, multiplication, and division take $\softO(\log N)$ bit operations in $\Z/N\Z$ and $\softO(n)$ basering operations in $R[x]/x^n\!R[x]$. 
In particular, this allows arithmetic operations in $\Qp$ to $n$ (relative) digits of precision, hereafter called field operations, in time $\softO(n \log p)$. 
Using Newton iteration, both square roots and the Teichm\"{u}ller character can be computed to $n$ digits of precision using $\softO(\log n)$ arithmetic operations. 
(We again consider only points in non-Weierstrass discs defined over unramified fields.)

\begin{proposition}\label{complexity}
Let $\int_P^Q \omega$ be a Coleman integral on a curve of genus $g$ over $\Q_p$, with $\omega=df_\omega+\sum_{i=1}^{2g-i}c_i\omega_i$ a differential of the second kind and with $P,Q$
in non-Weierstrass residue discs, defined over $\Q_p$, and accurate to $n$ digits of precision. Let $\Frob$ be the matrix of the action of Frobenius on the basis differentials, and let $m = v_p(\det(\Frob^t-I))$.  Let $F(n)$ be the running time of evaluating $f_\omega$ at $P$ and $Q$ to $n$ digits of precision. 
The value of the integral $\int_P^Q \omega$ can be computed to $n-\max\{m , \lfloor \log_p n \rfloor\}$ digits of precision in time $F(n) + \softO(pn^2g^2 + g^3n\log p)$. 
(Over a degree $N$ unramified extension of $\Q_p$, the analysis is the same with the runtime
multiplied by a factor of $N$.)
\end{proposition}

\begin{proof}
An essential input to the algorithm is the matrix of the action of Frobenius, which can be computed by Kedlaya's algorithm to $n$ digits of precision in running time $\softO(pn^2g^2)$. 
Inverting the resulting matrix can be (na\"ively) done with $O(g^3)$ arithmetic operations in $\Qp$. 
% asym fast MM -> precision loss?
It remains to be shown that no other step exceeds these running times. For the tiny integral on the first basis differential, the power series $x(t)/y(t) = x(t)f(x(t))^{-1/2}$ can be computed modulo $t^{n-1}$ using Newton iteration, requiring $\softO(n \log n)$ field operations. Each other basis differential can be computed from the first by multiplication by the linear polynomial $x(t)$ and the definite integral evaluated with $\softO(n)$ field operations, for a total of $\softO(gn^2)$ bit operations. 
Computing $\phi(P)$ and $\phi(Q)$ to $n$ digits of precision is cheap; directly using the formula in Algorithm~\ref{kedlaya} uses $\softO(g + \log p)$ field operations. 
The last potentially significant step is computing and evaluating the $f_i$ at each $P$ and/or $Q$. The coefficients of the $f_i$ can be read off in the reduction phase of Kedlaya's algorithm, and have $O(png)$ terms each. Evaluating (or even recording) all $g$ of these forms takes $\softO(png^2)$ field operations, or $\softO(pn^2g^2)$ bit operations, which is proportional to the cost of doing the reduction. 
\end{proof}

\subsection{Numerical examples}

Here are some sample computations made using our Sage implementation.
Additional examples will appear in the first author's upcoming PhD thesis.

\begin{example} 
Lepr\'evost \cite{leprevost} showed that the divisor $(1, -1) - \infty^+$ on the genus 2 curve
$y^2 = (2x-1)(2x^5 - x^4 - 4x^2 + 8x - 4)$ over $\Q$ is torsion of order 29. 
Consequently, the integrals of holomorphic differentials against this divisor must vanish.
We may observe this vanishing numerically, as follows. 
Let
$$C: y^2 = x^5 + \frac{33}{16}x^4 + \frac{3}{4}x^3 +\frac{3}{8}x^2 -\frac{1}{4}x +\frac{1}{16}$$ 
be the pullback of Leprevost's curve by the linear fractional transformation
$x \mapsto (1-2x)/(2x)$ taking $\infty$ to $1/2$.
The original points $(1,-1), \infty^+$ correspond to the points
$P = (-1,1)$,  $Q = (0,\frac{1}{4})$ on $C$. The curve $C$ has good reduction at $p=11$,
and we compute 
\[
\int_P^Q \omega_0 = 
 \int_P^Q \omega_1 = O(11^6),
 \int_P^Q \omega_2 = 7 \cdot 11 + 6 \cdot 11^{2} + 3 \cdot 11^{3} + 11^{4} + 5 \cdot 11^{5} + O(11^{6}),
\]
consistent with the fact that $Q-P$ is torsion and $\omega_0, \omega_1$ are holomorphic
but $\omega_2$ is not.
\end{example}

\begin{example} \label{chabauty}
We give an example arising from the Chabauty method, taken from \cite[\S~8.1]{mccallum-poonen}. 
Let $X$ be the curve $$y^2 = x(x-1)(x-2)(x-5)(x-6),$$
whose Jacobian has Mordell-Weil rank 1.
The curve $X$ has good
reduction at $7$, and
\[
X(\F_7) = \{(0,0), (1,0), (2,0), (5,0),
(6,0), (3,6), (3,-6), \infty\}.
\]
By
\cite[Theorem~5.3(2)]{mccallum-poonen}, we know $|X(\Q)| \leq 10$.
However, we can find 10 rational points on $X$: the
six rational Weierstrass points, and the points $(3, \pm 6),
(10, \pm 120)$. Hence $|X(\Q)| = 10$.

Since the Chabauty condition holds, there must exist a holomorphic
differential $\omega$ for which $\int_\infty^Q \omega = 0$ for all $Q \in X(\Q)$.
We can find such a differential by taking $Q$ to be one of the rational
non-Weierstrass points, then computing
$a := \int_{\infty}^{Q} \omega_0, b := \int_{\infty}^Q \omega_1$ and setting
$\omega = b\omega_0 - a\omega_1.$ For $Q = (3,6)$, we obtain
\begin{align*}a &= 6 \cdot 7 + 6 \cdot 7^{2} + 3 \cdot 7^{3} + 3 \cdot 7^{4} + 2 \cdot 7^{5} + O(7^{6}) \\
b &= 4 \cdot 7 + 2 \cdot 7^{2} + 6 \cdot 7^{3} + 4 \cdot 7^{5} + O(7^{6}).
\end{align*}
We then verify that $\int_Q^R \omega$ vanishes for each of the other rational points $R$.
\end{example}

\begin{remark}
It is worth pointing out some facts not exposed by Example~\ref{chabauty}. For instance,
since $\omega$ is already determined by a single rational non-Weierstrass point, we could
have used it instead of a brute-force seach to find other rational points.
More seriously, in other examples, the integral
$\omega$ may vanish at a point defined over a number field which has a rational multiple
in the Jacobian. Such points may be difficult to find by brute-force search; 
it may be easier to reconstruct them from $p$-adic approximations,
obtained by writing $\int_\infty^* \omega$
as a function of a linear parameter of a residue disc, then finding the zeroes
of that function.
\end{remark}

\section{Future directions}\label{sec6}

Here are some potential extensions of our computation of Coleman integrals.

\subsection{Iterated integrals}

Coleman's theory of integration is not limited to single integrals; it gives rise to 
an entire class of locally analytic functions, the \emph{Coleman functions},
on which antidifferentiation is well-defined. In other words, one can define integrals
\[
\int_P^Q \omega_n \cdots \omega_1
\]
which behave formally like iterated path integrals
\[
\int_0^1 \int_0^{t_1} \cdots \int_0^{t_{n-1}} f_n(t_n) \cdots f_1(t_1) \,dt_n \cdots\, dt_1.
\]
These appear in several applications of Coleman integration,
e.g., $p$-adic regulators in $K$-theory, and the nonabelian Chabauty method.

As in the case of a single integral,
one can use Frobenius equivariance to compute iterated Coleman integrals on hyperelliptic curves.
One obtains a linear system expressing all $n$-fold integrals of basis differentials in terms of lower order integrals. Note that the number of such $n$-fold integrals is
$(2g)^n$, so this is only feasible for small $n$. The cases $n \leq 4$ are already useful for
applications, but ideas for reducing the combinatorial explosion for larger
$n$ would also be of interest. (One must be slightly careful in dealing with
Weierstrass residue discs; see Remark~\ref{rem:iterated}.)

We have made some limited experiments with double Coleman integrals in Sage. The Fubini identity
\[
\int_P^Q \omega_2 \omega_1 + \int_P^Q \omega_1 \omega_2 = 
\left( \int_P^Q \omega_1 \right) \left( \int_P^Q \omega_2 \right)
\]
turns out to be a useful consistency check for both single and double integrals.

\subsection{Beyond hyperelliptic curves}

It should be possible to convert other algorithms for computing Frobenius actions on de Rham
cohomology, for various classes of curves, into algorithms for computing Coleman integrals
on such curves. Candidate algorithms include the adaptation of Kedlaya's algorithm to
superelliptic curves by Gaudry and G\"urel \cite{gaudry-gurel}, or the general algorithm
for nondegenerate curves due to Castryck, Denef, and Vercauteren \cite{castryck-denef-vercauteren}.
It should also be possible to compute Coleman integrals using Frobenius structures on Picard-Fuchs (Gauss-Manin) connections,
extending Lauder's \emph{deformation method} for computing Frobenius matrices
\cite{lauder-deformation}.

\subsection{Heights after Harvey}

We noted earlier that our algorithms for Coleman integration over $\Q_p$
have linear runtime dependence on the prime $p$, arising from the corresponding dependence
in Kedlaya's algorithm. In \cite{harvey:matfrob}, Harvey gives a variant of Kedlaya's
algorithm with only square-root dependence on $p$ (but somewhat worse dependence on other 
parameters), by reorganizing the computation so that the dominant step is finding the $p$-th
term of a linear matrix recurrence whose coefficients are polynomials in the sequence index. Harvey demonstrates
the practicality of his algorithm for primes greater than $2^{50}$, which may have some
relevance in cryptography for finding curves of low genus with nearly prime Jacobian orders.

It should be possible to use similar ideas to obtain
square-root dependence on $p$ for Coleman integration, 
by constructing a recurrence that computes not just the entries
of the Frobenius matrix but also the values $f_i(P)$ and $f_i(Q)$. 
However, this is presently a purely theoretical question, as 
we do not know of any applications of Coleman integration for very large $p$.

\bibliography{biblio}
\end{document}